\documentclass{amsart}
	
\usepackage[english]{babel}
\usepackage[latin1]{inputenc}
\usepackage{enumerate} 
\usepackage{latexsym} 
\usepackage{amsthm} 
\usepackage{amssymb} 
\usepackage{amsmath} 
\usepackage{verbatim} 
\usepackage{yfonts}
\usepackage[all,arc,curve,color,frame]{xy}
\usepackage{etoolbox} 

\usepackage{graphicx}

\usepackage{multirow} 
\usepackage{fixme}
	
\newtheorem{thm}{Theorem}[section]
\newtheorem{lem}[thm]{Lemma}

\newtheorem{prop}[thm]{Proposition}

\theoremstyle{definition}
\newtheorem{defi}[thm]{Definition}
\theoremstyle{remark}
\newtheorem{rema}[thm]{Remark}

\def\relay#1#2{%
  \expandafter\def\csname #1\endcsname{#2}%
}

\def\makecal#1{%
	\relay{c#1}{\ensuremath{\mathcal{#1}}}%
}
\newcommand{\makebb}[1]{\relay{bb#1}{\ensuremath{\mathbb{#1}}}}

\forcsvlist{\makecal}{X,Y,K,N,R,F,Q,P,U}
\forcsvlist{\makebb}{R,N,C,Q,D,Z,F,T}

\newcommand{\makemathop}[1]{\expandafter\DeclareMathOperator\expandafter{\csname #1\endcsname}{#1}}

\forcsvlist{\makemathop}{id, coker, ev}

\newcommand{\e}{\varepsilon}
\newcommand{\torus}{{\bbT}^2} 
\newcommand{\torusn}[1][n]{(\torus)^{[#1]}}

\title{Almost Commuting Orthogonal Matrices}

\author{Terry A Loring}
\address{University of New Mexico, Department of Mathematics and Statistics, Albuquerque, New Mexico 87131, USA}
\email{loring@math.unm.edu}
\thanks{The first named author thanks the Simons foundation (CGM 208723)}

\author{Adam P W S{\o}rensen}
\address{Department of Mathematical Sciences, University of Copenhagen, Universitetsparken 5,
2100 Copenhagen \O, Denmark}
\curraddr{School of Mathematics and Applied Statistics, University of Wollongong,
Wollongong, NSW 2522, Australia}
\email{apws@math.ku.dk}
\thanks{The second named author was supported by The Danish Council for Independent Research \textbar Natural Sciences}

\date{}

\begin{document}

\begin{abstract}
We show that almost commuting real orthogonal matrices are uniformly close to exactly commuting real orthogonal matrices. 
We prove the same for symplectic unitary matrices.
This is in contrast to the general complex case, where not all pairs of almost commuting unitaries are close to commuting pairs. 
Our techniques also yield results about almost normal matrices over the reals and the quaternions.
\end{abstract}

\maketitle

\section{Introduction}

In \cite{HalmosSomeUnsolved} Halmos asked if two almost commuting self-adjoint matrices are necessarily close to two exactly commuting self-adjoint matrices. 
To make this question as interesting as possible we take ``almost'' and  ``close'' to be uniform across all matrix sizes. 
The question was answered when Lin proved in \cite{LinAlmostCommuting} that indeed every pair of almost commuting self-adjoint matrices is always close to a pair of exactly commuting self-adjoint matrices. 
Shortly thereafter Friis and R{\o}rdam gave a short proof of Lin's Theorem in \cite{FriisRordamAlmostCommuting}. 

Before Lin's solution a lot of work went into investigating similar problems. 
Davidson showed in \cite{DavidsonAlmostCommuting} that triples of almost commuting self-adjoint matrices need not be close to exactly commuting triples.
Voiculescu showed that pairs of almost commuting unitary matrices are not necessarily close to pairs of exactly commuting unitaries in \cite{VoicuelscuAsymptoticallyCommuting}. 
Exel and the first named author gave a short proof of Voiculescu's result in \cite{ExelLoringAlmostCommuting}.
The main idea in \cite{ExelLoringAlmostCommuting} is that if $U,V$ are almost commuting unitaries then the winding number of the path in $\bbC \setminus \{0\}$ given by
\begin{align} \label{windingNumberPath}
	t \mapsto \det((1-t)UV + tVU), \quad t \in [0,1],
\end{align}
measures the obstruction to $(U,V)$ being close to commuting unitaries. 
The winding number for a commuting pair is zero so the winding number also has to be zero for any pair that can be perturbed to a commuting pair.  

The winding number is also referred to as the Bott index, a name that highlights its connection to $K$-theory. 
By \cite{ELPPushingForward,GongLinAlmost} an almost commuting commuting pair of unitary matrices $U,V$ is close to a commuting pair of unitary matrices if and only if the Bott index of the pair is zero. 
So long as we measure noncommutativity via the operator norm, use complex scalars, don't worry about algorithms or quantitative results, we can end here the story on almost commuting unitary matrices.

The recent focus in condensed matter physics on systems with time-reversal and other anti-unitary symmetries, see for instance \cite{ryu2010topological}, resulted in a new chapter of this story to be written.
The unitary matrices $U$ and $V$ in that arose in this context satisfied new relations  $U^{\tau}=U$ and $V^{\tau}=V$ (\cite{LorHastHgTe}), where $\tau$ is either the transpose or the dual operation $\tau=\sharp$ given by 
\[
\left(\begin{array}{cc}
A & B\\
C & D
\end{array}\right)^{\sharp}=\left(\begin{array}{cc}
D^{\mathrm{T}} & -B^{\mathrm{T}}\\
-C^{\mathrm{T}} & A^{\mathrm{T}}
\end{array}\right).
\]

Thus came an intense focus on symmetric unitary matrices and self-dual unitary matrices that almost commute. There arose an index, the Pfaffian-Bott index, that is linked to the spin Chern number in certain two-dimensional topological insulators. 
Building on the work in \cite{LorHastHgTe}, we characterized when such pairs of almost commuting self-dual or symmetric unitaries can be perturbed to an exactly commuting pair in \cite{LoringSorensenTimeReversal}. 
From a physics point of view, this was a natural place to look.

In the present paper we continue the story in a different direction, by asking when a pair of almost commuting real-valued unitaries matrices, i.e. real orthogonal matrices, are close to exactly commuting real-valued unitaries. 
Unlike in the complex case we find that a pair of almost commuting real orthogonals are always close to an exactly commuting pair. 
With very little extra work, we also get results for symplectic unitaries. 

For a pair of real orthogonal matrices the winding number of the path (\ref{windingNumberPath}) is always zero. 
By results from \cite{ELPPushingForward, GongLinAlmost} this means that close to any pair of almost commuting real orthogonal there is a pair of exactly commuting unitary (though not necessarily real orthogonal) matrices. 
Unlike in the case of self-dual unitaries, it turns out that there is no new obstruction to finding real orthogonal approximates.  
Thus, our main theorem is:	

\begin{thm} \label{mainthm:perturbation}
For any $\e > 0$ there exists a $\delta > 0$ such that whenever and $U,V$ are real orthogonal matrices in $M_n$ with 
\[
	\|UV - VU\| \leq \delta,
\]
there exist real orthogonal matrices $U',V' \in M_n$ such that
\[
	\|U - U'\|, \|V - V'\| \leq \e \quad \text{ and } \quad U'V' = V'U'.
\]
The identical statement holds in the symplectic unitary case.
\end{thm}

Following past work we will give a lifting solution to the perturbation problem. 
To properly state the relevant lifting problem we will use the theory of real $C^*$-algebras. 
We take the point of view that a real $C^*$-algebra is a $C^*$-algebra endowed with a map that acts like the transpose on matrices. 
We call such algebras $C^{*,\tau}$-algebras, and it is in the category of such algebras we prove lifting theorems.  

Throughout this paper we will use the terminology introduced in \cite{LoringSorensenLinsTheorem}.

Our main technical advance is Theorem~\ref{thm:realExtension} which concerns the interplay of ideals and the real structure.
When we have appropriate symmetry conditions Theorem~\ref{thm:realExtension} allows us to invoke complex lifting results in the real case. 
This way we avoid real $K$-theory, and so our techniques can be used in a fairly broad context.
What is critical is the absence of complex $K$-theoretic obstructions and that the underlying commuting situation corresponds to a two-dimensional CW complex with an involution that has only a zero- or one-dimensional set of fixed points.
We prove Theorem~\ref{mainthm:perturbation} by studying the torus with a specific rotation (see Definition~\ref{def:rotationReflection}). 
If we instead consider a disc with an involution that flips elements across the $x$-axis we obtain an additional real version of Lin's Theorem: 

\begin{thm} \label{thm:Lin's}
For any $\e > 0$ there exists a $\delta > 0$ such that whenever $X$ is a real matrix in $M_n$ with $ \| X \| \leq 1 $ and
\[
	\|XX^* - X^*X\| \leq \delta,
\]
there exists a normal real matrix $X' \in M_n$ such that $ \| X' \| \leq 1 $ and
\[
	\|X - X'\| \leq \e .
\]
The identical statement holds for quaternionic matrices.
\end{thm}

Commutativity is just one relation for which we can examine the approximate version. 
We can also ask about almost anti-commuting unitary matrices, for example.
Both of these questions can be studied by looking at the class of $C^{*}$-algebras called noncommutive CW complexes (NCCW complexes).
Glossing over the technical details, the general story about NCCW complexes and approximate relations in is as follow:
If relations corresponds to a one- or two-dimensional NCCW complex, then complex matrices that almost satisfy these relations can be perturbated to ones that exactly satisfy the relations, if and only if the invariants coming from the $K$-theory of the 2-cells vanish (\cite{ELPPushingForward}).

Keeping the connection to physics in mind we might wish to define a real $C^*$-algebra version of NCCW complexes and use them to study perturbation problems for matrices.
The number of possible generalization here is surprisingly large.
An example of a two-cell in an NCCW complex is $C(\bbD) \otimes M_n(\bbC)$, where $\bbD$ denotes the unit disc.
To put an involution on this we replace $M_n(\mathbb{C})$ with a simple, finite-dimensional $C^{*,\tau}$-algebra, so one of
\[
	(M_{n}(\mathbb{C}),T ), \quad (M_{2n}(\mathbb{C}), \sharp), \quad \text{ or } \quad (M_{n}(\mathbb{C})\oplus M_{n},(\mathbb{C}),\zeta)
\]
where $\zeta(A,B)=(B^T,A^T)$.
If one prefers to think of real $C^*$-algebras not as $C^*$-algebras with extra structure, but as algebras over $\bbR$, then this is equivalent to choosing 
\[
	M_{n}(\mathbb{R}), \quad M_{n}(\mathbb{H}), \quad \text{ or } \quad M_{n}(\mathbb{C}),
\]
as our finite dimensional algebra. 
We must also put an involution on the disc and there are at least three options; the fixed points can be a single point, a line, or the whole space.
So even in this simple case, we are staring at $8$ generalizations.  Theorem~\ref{thm:Lin's} is a basic theorem about one of these generalized
two-cells.
More examples are needed to guide us safely forward.

We believe that our techniques can be useful in a more general study of real NCCW complexes, whatever the final definition, and while we chose our main theorems for their nice nice and easy statements in linear algebra, they can act as useful examples of the behavior we want real NCCW complexes to display.

\section{Recasting the problem}

Denote by $\torus$ the two-torus and recall that $C(\torus)$ is the universal $C^*$-algebra generated by two commuting unitaries.

\begin{defi} \label{def:rotationReflection}
Let $\tau$ be the unique reflection on $C(\torus)$ such that $u^\tau = u^*$ and $v^\tau = v^*$, where $u,v$ are the universal unitaries generating $C(\torus)$.
We call this reflection the \emph{the rotation reflection} 
\end{defi}

\begin{figure}
\begin{center}
\includegraphics[clip,scale=1.5]{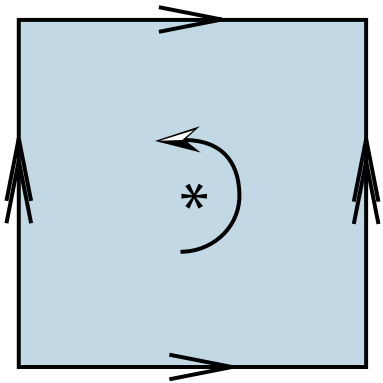}
\caption{
The torus with involution that rotates the usual square model by 180 degrees.
The left and right edges are identified, as are the top and bottom, so there is
a second fixed point represented by any of the corners in the model.
\label{fig:torus_tau} 
}
\end{center}
\end{figure}

The rotation reflection can be described topologically by thinking of the torus as the square with opposite sides identified. 
Then the reflection is simply rotation around the center point by $180$ degrees,
as illustrated in Figure~\ref{fig:torus_tau}.

The lifting theorem we aim to prove is the following:

\begin{thm} \label{mainthm:lifting}
Let $(d_n)$ be a sequence of natural numbers, and let $\tau$ be the rotation reflection on $C(\torus)$. 
Let $\tau_0$ denote either the transpose operation or the dual operation on matrices, the latter only allowed if all the $d_n$ are even.
For every $*$-$\tau$-homomorphism 
\[
\phi \colon C(\torus, \tau) \to \prod (M_{d_n}, \tau_0) / \bigoplus (M_{d_n}, \tau_0)
\]
there exists a $*$-$\tau$-homomorphism $\psi \colon C(\torus, \tau) \to \prod (M_{d_n}, \tau_0)$ such that 
\[
	\pi \circ \psi = \phi,
\]
where $\pi \colon \prod M_{d_n} \to \prod M_{d_n} / \bigoplus M_{d_n}$ is the quotient map.   
\end{thm}  

The equivalence of Theorem \ref{mainthm:perturbation} and Theorem \ref{mainthm:lifting} is proved by well-known methods.

Our strategy for proving Theorem \ref{mainthm:lifting} follows the beaten path. 
First we will ``open up holes'' in the torus to produce an
``$n$-holed torus'' with involution, which then we will retract onto a one dimensional CW complex that looks like a grid, again with an involution. 
Past works have had the advantage of knowing that the relevant one dimensional CW complex are semipropjective. 
In our case, we do not know that, so we proceed as Friis and R{\o}rdam and ``open up holes'' in the grid, then we retract onto a finite set of points. 
Now we can use semiprojectivity to solve our lifting problem.   

\section{From the complex to the real case} 

The aim of this section is to develop a method to move results on extending $*$-homomorphisms from the category of $C^*$-algebras to the category of $C^{*,\tau}$-algebras.
Our main tool will be the following theorem. 

\begin{thm} \label{thm:realExtension}
Let $\theta \colon A \to A_1$ be a proper $*$-homomorphism. 
Suppose we have the following commutative diagram of $C^{*,\tau}$-algebras with exact rows:
\[
	\xymatrix{
		0 \ar[r] & A_1 \oplus A_1^{\mathrm{op}} \ar[r]^-{\iota_1} & B_1 \ar[r] & D \ar[r] & 0 \\
		0 \ar[r] & A \oplus A^{\mathrm{op}} \ar[u]^{\bar{\theta}} \ar[r]_-{\iota} & B \ar[u]^{\beta} \ar[r] & D \ar@{=}[u] \ar[r] & 0
	}
\]
where the $\tau$-operation on $A \oplus A^{\mathrm{op}}$ and $A_1 \oplus A_1^{\mathrm{op}}$ flips the summands, and $\bar{\theta}$ is given by $\bar{\theta}(a_1, a_2) = (\theta(a_1), \theta(a_2))$.
Given a $*$-$\tau$-homomorphism $\phi \colon B \to E$ there exists a $*$-$\tau$-homomorphism $\psi \colon B_1 \to E$ such that $\psi \circ \beta = \phi$ if there exists a $*$-homomorphism $\lambda \colon A_1 \to E$ with $\lambda(\theta(a)) = \phi(\iota(a,0))$ for all $a \in A$. 
\end{thm}

The proof requires two auxiliary results. 

\begin{lem} \label{opSumHom}
Suppose $\theta \colon A \to B$ is a $*$-homomorphism, then  $\bar{\theta} \colon A \oplus A^{\mathrm{op}} \to B \oplus B^{\mathrm{op}}$ given by $\bar{\theta}(a_1, a_2) = (\theta(a_1), \theta(a_2))$ is a $*$-$\tau$-homomorphism, when both $A \oplus A^{\mathrm{op}}$ and $B \oplus B^{\mathrm{op}}$ are given the $\tau$-operation that flips the two summands. 
Furthermore, if $\theta$ is proper then $\bar{\theta}$ is proper. 
\end{lem}
\begin{proof}
It is well known that $\theta$ defines a $*$-homomorphism from $A^{\mathrm{op}}$ to $B^{\mathrm{op}}$, so $\bar{\theta}$ is the direct sum of two $*$-homomorphisms and hence a $*$-homomorphism. 
As the $\tau$-operation on $A \oplus A^{\mathrm{op}}$ and $B \oplus B^{\mathrm{op}}$ is essentially the same we also have that $\bar{\theta}$ preserves the $\tau$ operation. 

To see that $\bar{\theta}$ is proper let $(b_1,b_2) \in B \oplus B^{\mathrm{op}}$. 
Since $\theta$ is proper we can find elements $a_1, a_2 \in A$ and $x_1, x_2 \in B$ such that 
\[
	b_1 = \theta(a_1)x_1, \quad \textrm{and,} \quad b_2 = x_2 \theta(a_2).
\]
by \cite[\S 4]{Pedersen-Factorization}.  We see that 
\[
	\bar{\theta}(a_1, a_2)(x_1, x_2) = (\theta(a_1), \theta(a_2))(x_1, x_2) = (\theta(a_2)x_1, x_2 \theta(a_2)) = (b_1, b_2).
\]
Therefore, $\bar{\theta}(A \oplus A^{\mathrm{op}})(B \oplus B^{\mathrm{op}}) = B \oplus B^{\mathrm{op}}$, so $\bar{\theta}$ is proper. 
\end{proof}

\begin{prop} \label{opSumExtension}
Suppose $\theta \colon A \to B$ is a proper $*$-homomorphism, and let $\bar{\theta}$ be the $*$-$\tau$-homomorphism constructed in Lemma \ref{opSumHom}. 
Let $D$ be a $C^{*,\tau}$-algebra and let $\phi \colon A \oplus A^{\mathrm{op}} \to D$ be a $*$-$\tau$-homomorphism.
If there is $*$-homomorphism $\gamma \colon B \to D$ such that $\gamma(\theta(a)) = \phi(a,0)$ for all $a \in A$, then there is a $*$-$\tau$-homomorphism $\psi \colon B \oplus B^{\mathrm{op}} \to D$ such that $\psi \circ \bar{\theta} = \phi$.
\end{prop}

If we define a $*$-homomorphism $\tilde{\phi} \colon A \to D$ by $\tilde{\phi}(a) = \phi(a,0)$, then Proposition \ref{opSumExtension} states that if the diagram below on the left can be completed then so can the diagram on the right
\[
	\xymatrix{
		B \ar@{-->}[dr]^{\gamma} & \\
		A \ar[u]^{\theta} \ar[r]_{\tilde{\phi}} & D
	}
	\quad \quad \quad \quad \quad
	\xymatrix{
		B \oplus B^{\mathrm{op}} \ar@{-->}[dr]^{\psi} & \\
		A \oplus A^{\mathrm{op}} \ar[u]^{\bar{\theta}} \ar[r]_{\phi} & D 
	}
\]

\begin{proof}
We first show that the image of $\gamma$ is orthogonal to its reflection under $\tau$. 
So let $b_1,b_2 \in B$, and use the properness of $\theta$ to find $a_1, a_2 \in A$ and $x_1, x_2 \in B$ such that $b_i =  x_i\theta(a_i)$, $i=1,2$. 
Then we have that 
\begin{align*}
	\gamma(b_1)\gamma(b_2)^\tau	& =	\gamma(x_1\theta(a_1)) \gamma(x_2\theta(a_2))^\tau \\
		& = \gamma(x_1) \gamma(\theta(a_1)) \gamma(\theta(a_2))^\tau \gamma(x_2) \\
		& =	\gamma(x_1) \phi((a_1, 0)) \phi((a_2, 0))^\tau \gamma(x_2) \\
		& =	\gamma(x_1) \phi((a_1,0)(0,a_2)) \gamma(x_2) = 0, 
\end{align*}
where we used that $\phi$ is a $\tau$-preserving homomorphism. 

We define a $\tau$-preserving map $\psi \colon B \oplus B^{\tau} \to D$ by
\[
	\psi(b_1,b_2) = \gamma(b_1) + \gamma(b_2)^\tau.
\]
Clearly $\psi$ is linear and $*$-preserving.  Since $\gamma(B) \perp \gamma(B)^\tau$ we also have that $\psi$ is multiplicative, and therefore $\psi$ is a $*$-homomorphism.

Finally we check that $\psi \circ \bar{\theta} = \phi$. 
Let $(a_1, a_2) \in A \oplus A^{\mathrm{op}}$, then 
\begin{align*}
	\psi(\bar{\theta}(a_1, a_2))	& =	\psi((\theta(a_1), \theta(a_2))) = \gamma(\theta(a_1)) + \gamma(\theta(a_2))^\tau \\
		& =	\phi((a_1, 0)) + \phi((a_2, 0)^\tau) = \phi(a_1, a_2). \qedhere
\end{align*} 
\end{proof}

\begin{proof}[Proof of Theorem \ref{thm:realExtension}]
Since $\bar{\theta}$ is proper (Lemma \ref{opSumHom}) the left most square in the diagram is a pushout by \cite[Theorem 5.4]{LoringSorensenTimeReversal}, so it suffices to find a $*$-$\tau$-homomorphism $\chi \colon A_1 \oplus A_1^{\mathrm{op}} \to E$ such that 
\[
	\xymatrix{
		A_1 \oplus A_1^{\mathrm{op}} \ar@/^1em/[drr]^{\chi} & & \\
		A \oplus A^{\mathrm{op}} \ar[u]^{\bar{\theta}} \ar[r]_-{\iota} & B \ar[r]_{\phi} & E.
	}
\]
commutes. 
By Proposition \ref{opSumExtension} we can find such a $\chi$ since we have a map $\lambda \colon A_1 \to E$ such that 
\[
	\xymatrix{
		A_1 \ar@/^1em/[drr]^{\lambda} & & \\
		A \ar[u]^{\theta} \ar[r] & B \ar[r]_{\phi} & E.
	}
\]
commutes. 
\end{proof}

The usefulness of Theorem \ref{thm:realExtension} depends on the quality of our extensions results for $C^*$-algebras.
Since we aim to ``punch holes'' in a torus and a grid, we need both two and one dimensional extensions results for $C^*$-algebra. 
We prove such results next. 
The flavor of these results is certainly not new, but we need to do a little work to get them in exactly the form we want. 

\subsection{Two dimensional case}

The first criteria for using Theorem \ref{thm:realExtension} is that we have a proper map. 

\begin{lem} \label{lem:twoDimAlphaIsProper}
Let $\cU$ be the open unit disc in $\bbC$ and let $A$ be the half open annulus, that is  $A = \{ z \in \bbC \mid 1 \leq |z| < 2 \}$.  Let $\alpha \colon C_0(\cU) \to C_0(A)$ be the map that is given by collapsing the inner circle of $A$ to a single point.
Then $\alpha$ is proper. 
\end{lem}
\begin{proof}
A strictly positive element in a commutative $C^*$-algebra is simply a function that only takes strictly positive values. 
Since applying $\alpha$ to a function does not change what values the function takes $\alpha$ is proper. 
\end{proof}

The next requirement of Theorem \ref{thm:realExtension} is that we have an extendable $*$-homomor-phism. 
The result we want closely related to results in \cite{LoringWhenMatrices}, to get it in exactly the form we need, we massage \cite[Theorem 9]{LoringWhenMatrices}.

\begin{prop} \label{prop:twoDimExtension}
Let $\alpha \colon C_0(\cU) \to C_0(A)$ be as in Lemma \ref{lem:twoDimAlphaIsProper}. 
Given a sequence of natural numbers $(d_n)$ and a $*$-homomorphism $\phi \colon C_0(\cU) \to \prod M_{d_n} / \bigoplus M_{d_n}$ there exists a $*$-homomorphism $\psi \colon C_0(A) \to \prod M_{d_n} / \bigoplus M_{d_n}$ such that $\psi \circ \alpha = \phi$ if $K_*(\phi) = 0$.
\end{prop}
\begin{proof}
To simplify notation we let $Q = \prod M_{d_n} / \bigoplus M_{d_n}$.
Since $C_0(\cU)$ and $C_0(A)$ are non-unital and $Q$ is unital, proving the proposition is equivalent to proving the same extension property for $\tilde{\alpha} \colon C(S^2) \to C(\bbD)$, where $\bbD$ is the closed unit disc. 
The map $\tilde{\alpha}$ is given by collapsing the boundary of the disc to a single point. 
Pictorially our extension problem is
\[
\xymatrix{
	C(\bbD) \ar@{-->}[dr]^-{\psi} & \\
	C(S^2) \ar[u]^-{\tilde{\alpha}} \ar[r]_-{\tilde{\phi}} & Q
}
\]
The $K$-theory condition becomes that $K_0(\tilde{\phi})$ kills the so-called Bott element in $K_0(C(S^2))$ (see \cite[Section 9.2.10]{BlackadarKTheoryBook} for a description of the generator of $K_0(C_0(\bbR^2))$, its image in $K_0(C(S^2))$ is the Bott element).

Denote by $X_1$ the cylinder $\{ r e^{2 \pi i \theta} \mid 1 \leq r \leq 2 \}$.
We have a map $\gamma \colon C(S^2)\to C(X_1)$ that is given  by collapsing the top and bottom of the cylinder to two points, and a map $\beta \colon C(\bbD) \to C(X_1)$ that is given by collapsing the top of the cylinder to one point. 
Notice that $\gamma = \beta \circ \tilde{\alpha}$, that is the following diagram commutes.
\[
	\xymatrix{
		C(S^2) \ar[r]_{\tilde{\alpha}} \ar@/^1em/[rr]^{\gamma} & C(\bbD) \ar[r]_{\beta} & C(X_1) 
	}
\]
We see that it suffice to extend $\phi$ to $C(X_1)$ and then pre-compose that extension with $\beta$. 
By \cite[Theorem 9]{LoringWhenMatrices} such an extension will exists, if two specified projections, say $p,q$, are Murray-von Neumann equivalent. 
Since $Q$ has stable rank one it has the cancellation property (\cite[Proposition 6.5.1]{BlackadarKTheoryBook}, so $p$ and $q$ are Murray-von Neumann equivalent if they represent the same class in $K$-theory. 
By inspection of the definition of the Bott element and of $p$ and $q$, we see that this happens if $K_0(\tilde{\phi})$ kills the Bott element. 
\end{proof}

\subsection{One dimensional case}

In the one dimensional case we also need to know that a certain map is proper. 
The proof of that is very similarly to the proof of Lemma \ref{lem:twoDimAlphaIsProper}, so we omit it. 

\begin{lem} \label{lem:alphaIsProper}
Let $\alpha \colon C_0((0,1)) \to C_0((0,1] \cup [2,3))$ be the map that is given by identifying the endpoints of the half-open intervals. 
Then $\alpha$ is proper. 
\end{lem}

This time our desired extension result is a special case of \cite[Lemma 3.2]{ELP-Fragility}, so we wont give a proof. 
It is however worth noting, that a very short proof might be to simply observe that every unitary in $\prod M_{d_n} / \bigoplus M_{d_n}$ has a logarithm.

\begin{prop} \label{prop:oneDimExtension}
Let $\alpha \colon C_0((0,1)) \to C_0((0,1] \cup [2,3))$ be as in Lemma \ref{lem:alphaIsProper}. 
Suppose we are given a sequence of natural numbers $(d_n)$ and a $*$-homomorphism $\phi \colon C_0((0,1)) \to \prod M_{d_n} / \bigoplus M_{d_n}$ then there exists a $*$-homomorphism $\psi \colon C_0((0,1] \cup [2,3)) \to \prod M_{d_n} / \bigoplus M_{d_n}$ such that $\psi \circ \alpha = \phi$.
\end{prop}

%
%

\section{Opening up holes}  \label{sec:make_holes}

We will now use Theorem \ref{thm:realExtension} to open up holes in the torus (and incidentally other two dimensional CW complexes) and a grid. 
The main idea is that to remain in the category of $C^{*,\tau}$-algebras we do not open up holes one at a time, rather we open them up in pairs.

\subsection{Two dimensional case}

We wish to use Proposition \ref{prop:twoDimExtension} to guarantee the existence of $C^*$-extensions of our maps. 
To this end we need to know that the $K$-theory of certain maps vanish. 
This is the content of Proposition \ref{prop:noKTheory}.
We will use the connection discovered in \cite{ExelLoringInvariants} between the winding number invariant, first seen in \cite{ExelLoringAlmostCommuting}, and the $K$-theory of maps out of $C(\torus)$. 

\begin{defi}[{\cite[Defintion 1.3]{ExelLoringInvariants}}]
Let $U,V \in M_n$ be two unitary matrices.
If $\|UV - VU\| < 2$ we define their \emph{Bott index} (also called \emph{winding number}), as the winding number of the path
\[
	[0,1] \ni t \mapsto \det((1-t)UV + tVU). 
\]
\end{defi}

\begin{lem} \label{lem:trivialBottIndex}
Let $U,V \in M_n$ be uniatry matrices that are either real or symplectic.
When $\| UV - VU \|$ is sufficiently small the Bott index of $(U,V)$ is zero. 
\end{lem}
\begin{proof}
In the real case we have that $\det((1-t)UV + tVU)$ is real for all $t \in [0,1]$. 
This clearly implies that the winding number of the path is zero.


In the symplectic unitary case we use an alternate formula for the Bott index,
\[
\frac{1}{2\pi i} \mathrm{Tr} \left( \mathrm{Log} \left( VUV^*U^*\right)\right)
\]
from \cite{ExelSoftTorus}, which is valid for small commutators.  
The spectral theorem \cite[Theorem 2.4]{LoringQuaternions} for a normal matrix $X$ for which $X^\sharp = X^*$ states that all its eigenvalues will appear in conjugate pairs $e^{i\theta}, e^{-i\theta}$. 
So the eigenvalues of the principal logarithm of $VUV^*U^*$ appear in pairs $i\theta, -i\theta$, and hence the  principal logarithm will have trace zero.
\end{proof}

\begin{prop} \label{prop:noKTheory}
Let $(d_n)$ be a sequence of natural numbers, let $\phi \colon C(\torus) \to \prod M_{d_n} / \bigoplus M_{d_n}$ be a $*$-homomorphism, and let $\iota \colon C_0(\cU) \to C(\torus)$ be an inclusion.  
Denote by $\tau$ the reflection on $\prod M_{d_n} / \bigoplus M_{d_n}$ induced either by the transpose or the dual map on all the $M_{d_n}$.
Let $u,v$ be the universal generators of $C(\torus)$.
If $\phi(u)^\tau = \phi(u)^*$ and $\phi(v)^\tau  = \phi(v)^*$, then $K_0(\phi \circ \iota) = 0$.
\end{prop}

\begin{proof}
It is well known that $K_0(C(\torus))$ is $\bbZ \oplus \bbZ$, where the class of the constant function $1$ is $(1,0)$ and the class of the so-called Bott projection is $(1,1)$. 
For a projection $P \in M_n(C(\torus))$, the value of $[P]_{K_0}$ in the first summand is the trace of $P$.
The trace on $C(\torus)$ is the usual matrix trace composed with integration. 

In \cite[9.2.10]{BlackadarKTheoryBook} a generator of $K_0(C_0(\cU)) \cong \bbZ$ is given. 
One sees that it has trace zero, so that when we use $\iota$ to move it into $C(\torus)$ it will still have trace zero. 
We can then complete our proof, by showing that $\phi_* \colon K_0(C(\torus)) \to K_0(\prod M_{d_n} / \bigoplus M_{d_n})$ kills the second summand of $K_0(C(\torus))$. 
This amounts to showing $\phi(1)$ is $K$-theory equivalent to $\phi^{(2)}(\beta)$, where $\beta$ denotes the Bott projection in $M_2(C(\torus))$. 

\begin{figure}
\begin{center}
\includegraphics[clip,scale=1.5]{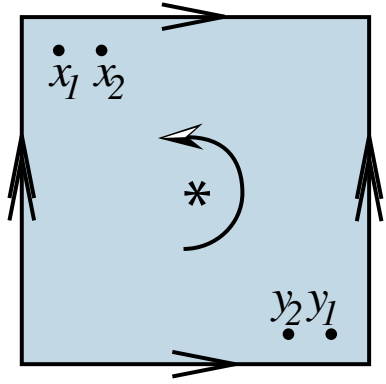}
\caption{
Choosing points on the torus in pairs so the involution sends $x_j$ to
$y_j$.
\label{fig:torusWithPoints}
}
\end{center}
\end{figure}

Following \cite[section 1]{ExelLoringInvariants} we define for any two unitaries $U,V$ in a $C^*$-algebra $A$ an element $e(U,V) \in M_2(A)$, such that $e(U,V)$ is a projection when $U$ and $V$ commute, and is close to a projection when $U$ and $V$ almost commute. 
Furthermore, we have $e(\phi(u),\phi(v)) = \phi^{(2)}(\beta)$.
By \cite[Example 2.21]{LoringSorensenLinsTheorem} we can find unitary matrices $U_n, V_n \in M_{d_n}$ such that $(U_n), (V_n) \in \prod M_{d_n}$ are lifts of $\phi(u)$ and $\phi(v)$, and such that $U_n^\tau = U_n^*$ and $V_n^\tau = V_n^*$. 
Since $\phi(u)$ and $\phi(v)$ commute, we can assume that $U_n$ and $V_n$ commute as well as we would like. 
Hence, Lemma \ref{lem:trivialBottIndex} tells us that the Bott index of any pair $(U_n, V_n)$ is zero. 
By \cite[Theorem 4.1]{ExelLoringInvariants} we then have that the so called $K$-theory invariant of $(U_n, V_n)$ is zero for all $n$. 
This implies that for a suitable indicator function $\chi$, we have that $\chi(e(U_n,V_n))$ is Murray-von Neumann equivalent to the unit of $M_2(M_{d_n})$ for all $n$. 
Denote the quotient map by $\pi \colon \prod M_{d_n} \to \prod M_{d_n} / \bigoplus M_{d_n}$.
Then we have 
\begin{align*}
	\phi^{(2)}(\beta) &= e(\phi(u),\phi(v)) = \pi^{(2)}((\chi(e(U_n,V_n)))) \\
					  &\sim_{MvN} \pi^{(2)}\left(\begin{pmatrix} 1 & 0 \\ 0& 0 \end{pmatrix} \right) = \phi^{(2)}\left(\begin{pmatrix} 1 & 0 \\ 0& 0 \end{pmatrix} \right),
\end{align*}
where $\sim_{MvN}$ denote Murray-von Neumann equivalence. 
So $\phi^{(2)}(\beta)$ represents the same $K$-theory class as $\phi(1)$.
\end{proof}

Since there is no $K$-theory obstruction we can always ``open up holes'' in our torus. 
Denote by $\torusn$ the $n$-times perforated torus, as shown
in Figure~\ref{fig:nHoleTorus}.

\begin{thm} \label{thm:twoDimHoles}
Let $\tau$ be the rotation reflection on the torus described in Definition \ref{def:rotationReflection}. 
Suppose that $\{x_1,x_2, \ldots, x_n, y_1, y_2, \ldots, y_n \}$ are $2n$ distinct points in the interior of the two cell of the torus chosen so that $\tau(x_i) = y_i$, for $1 \leq i \leq n$.
Let $\torusn[2n]$ be the $2n$-times perforated torus, where we have opened up holes at our $2n$ points. 
By the choice of points, $\tau$ will also define a reflection on $\torusn[2n]$, and the map $\alpha^{[2n]} \colon C(\torus, \tau) \to C(\torusn[2n], \tau)$, given by collapsing the holes to points, will be a $*$-$\tau$-homomorphism.
Given a sequence of natural numbers $(d_n)$ and a $*$-$\tau$-homomorphism $\phi \colon C(\torus, \tau) \to \prod (M_{d_n}, \tau_0) / \bigoplus (M_{d_n}, \tau_0)$, where $\tau_0$ is the transpose or dual map, there exists a $*$-$\tau$-homomorphism $\psi \colon C(\torusn[2n],\tau) \to \prod (M_{d_n}, \tau_0)$ such that $\psi \circ \alpha^{[2n]} = \phi$. 
\end{thm}

\begin{figure}
\begin{center}
\includegraphics[clip,scale=1.5]{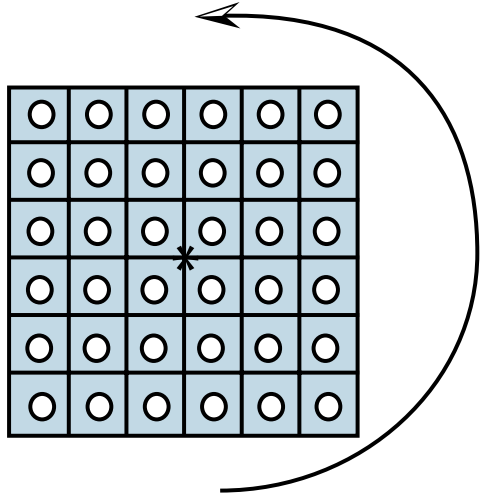}
\caption{
The result of perforating the torus in pairs so the involution lifts to
the resulting space.
\label{fig:nHoleTorus}
}
\end{center}
\end{figure}

\begin{proof}
It is clear that $\tau$ defines a reflection on $C(\torusn[2m])$ for all $1 \leq m \leq n$ and that the $*$-homomorphisms $\alpha^{[2m]} \colon C(\torus, \tau) \to C(\torusn[2m], \tau)$ given by collapsing the holes to points will all be $\tau$-preserving. 

We will open up two holes at a time.
We begin by opening holes at $x_1$ and $y_1$. 
Let $A$ denote the half open annulus, let $\cU$ be the open unit disc, and let $\theta$ be the map that collapses the inner circle of the annulus to a point, all as in Lemma \ref{lem:twoDimAlphaIsProper}. 
Denote by $\overline{\theta}$ the map $\theta \oplus \theta^{\mathrm{op}} = \theta \oplus \theta$. 
We have the following commutative diagram, of $C^*$-algebras, with exact rows
\[
	\xymatrix{
		0 \ar[r] & C_0(A) \oplus C_0(A) \to \ar[r]^-{\iota_2} & C(\torusn[2]) \ar[r] & D \ar[r] & 0 \\
		0 \ar[r] & C_0(\cU) \oplus C_0(\cU) \ar[u]^{\bar{\theta}} \ar[r]_-{\iota} & C(\torus) \ar[u]^{\alpha^{[2]}} \ar[r] & D \ar@{=}[u] \ar[r] & 0
	}
\]
The ideal inclusions put small open discs around the chosen points and the holes, respectively. 
If we give the ideals the reflection that flips the two summands we get a commutative diagram of $C^{*,\tau}$-algebras. 
By Lemma \ref{lem:twoDimAlphaIsProper} $\theta$ is proper, so by Theorem \ref{thm:realExtension} we can extend $\phi$ if we can extend $\phi \circ \iota$ restricted to one summand to $C_0(A)$. 
This in turn can be done by Proposition \ref{prop:twoDimExtension} since the relevant $K$-theory vanishes by Proposition \ref{prop:noKTheory}. 

Suppose now we have already opened up $2k$ holes for some $k \geq 1$, and denote the extension of $\phi$ to $C(\torusn[2k])$ by $\phi_{2k}$.
We can extend $\phi_{2k}$, and hence $\phi$, to $C(\torusn[2(k+1)])$ using the same techniques we used to open the first two holes, if we can show that $\phi_{2k} \circ \iota_{2k}$ has trivial $K$-theory when restricted to one summand.  
Define an inclusion $\kappa \colon C_0(\cU) \to C(\torus)$ such that the following diagram commutes
\[
	\xymatrix{
		& C(\torusn[2k]) \ar[dr]^{\phi_{2k}} & \\
		C_0(\cU) \ar[ur]^{\iota_{2k}} \ar[dr]_{\kappa} & & \prod (M_{d_n}, \tau) / \bigoplus (M_{d_n}, \tau) \\
		& C(\torus) \ar[ur]_{\phi} \ar[uu]^{\alpha^{[2k]}}&
	}
\]
where we have abused notation slightly, and used $\iota_{2k}$ to denote the restriction of $\iota_{2k}$. 
By proposition \ref{prop:noKTheory} we have that $K_0(\phi \circ \kappa) = 0$, so 
\[
	K_0(\phi_{2k} \circ \iota_{2k}) = K_0(\phi_{2k} \circ \alpha^{[2k]} \circ \kappa) = K_0(\phi \circ \kappa) = 0.
\]
Hence we can continue to extend for as many steps as we want.
\end{proof}

\begin{rema} \label{rem:anyTwoDim}
Note that in the proof of Theorem \ref{thm:twoDimHoles} we only used at two points that we were dealing with the torus. 
Most importantly, we used  that the torus is a two dimensional CW complex, so that we could make sense of ``opening up holes'' in it. 
Secondly, we applied Proposition \ref{prop:noKTheory} to see that the $K$-theory of the relevant $*$-homomorphisms vanished. 
Hence the techniques outlined in the proof can be used to ``open up holes'' in other two dimensional CW complexes, so long as we know that the $*$-homomorphisms we want to extend have trivial $K$-theory. 
\end{rema}

\subsection{One dimensional case}

\begin{thm} \label{thm:oneDimHoles}
Let $X$ be a one dimensional CW complex, and let $\tau$ be a reflection on $X$. 
Suppose that $\{x_1,x_2, \ldots, x_n, y_1, y_2, \ldots, y_n \}$ are $2n$ distinct points in the interior of $1$-cells of X chosen so that $\tau(x_i) = y_i$, for $1 \leq i \leq n$.
Let $X^{[2n]}$ be as $X$ but where we have opened up gaps at our $2n$ points. 
By the choice of points, $\tau$ will also define a reflection on $X^{[2n]}$, and the map $\alpha^{[2n]} \colon C(X, \tau) \to C(X^{[2n]}, \tau)$, given by collapsing the gaps to points, will be a $*$-$\tau$-homomorphism.
Given a sequence of natural numbers $(d_n)$ and a $*$-$\tau$-homomorphism $\phi \colon C(X, \tau) \to \prod (M_{d_n}, \tau_0) / \bigoplus (M_{d_n}, \tau_0)$, where $\tau_0$ is the transpose or dual map, there exists a $*$-$\tau$-homomorphism $\psi \colon C(X^{[2n]},\tau) \to \prod (M_{d_n}, \tau_0)$ such that $\psi \circ \alpha^{[2n]} = \phi$. 
\end{thm}

The proof of Theorem \ref{thm:oneDimHoles} is very similar to the proof of Theorem \ref{thm:twoDimHoles}, so we will skip it. 
The main differences are that we refer to Lemma \ref{lem:alphaIsProper} instead of lemma \ref{lem:twoDimAlphaIsProper}, to Proposition \ref{prop:oneDimExtension} instead of Proposition \ref{prop:twoDimExtension}, and that we do not have to worry about $K$-theory.

\section{Proofs of the main theorems}

We will now prove out main theorems. 
Our strategy of opening up holes and retracting to lower dimensional spaces follows past work, see for instance the proof of \cite[Theorem 19.2.7]{LoringBook}.

Consider the torus with its usual CW structure, that is a square where we identify opposite edges. 
Replace the two cell in the CW structure for the torus with an evenly spaces grid with $2n$ holes, and call the resulting one dimensional CW complex $\Gamma_{2n}$.
We can think of $\Gamma_{2n}$ as a ``fishnet torus''. 

\begin{proof}[Proof of Theorem \ref{mainthm:lifting} (and therefore Theorem \ref{mainthm:perturbation})]
Suppose $(d_n) \subseteq \bbN$ and 
\[
	\phi \colon C(\torus, \tau) \to \prod (M_{d_n}, \tau_0) / \bigoplus (M_{d_n}, \tau_0),
\]
are given. 
To ease notation let $(Q,\tau) = \prod (M_{d_n}, \tau_0) / \bigoplus (M_{d_n}, \tau_0)$.
By techniques similar to those used in \cite[Theorem 3.3]{EilersLoringComputingContingencies} we see that it suffices to find an approximate lift for a finite generating set of $C(\torus)$. 
We think of the torus as a square with opposing sides identified and pick, for some large $n$, $2n$ evenly distributed distinct points $\{x_1, x_2, \ldots, x_n, y_1, y_2, \ldots, y_n \}$ such that $\tau(x_i) = y_i$, and such that the points are the center points of the grid $\Gamma_{2n}$.

By Theorem \ref{thm:twoDimHoles} we can extend $\phi$ to a $*$-$\tau$-homomorphism $\psi \colon C(\torusn[2n], \tau) \to (Q,\tau)$. 
Choosing $n$ big enough we only make a small mistake, when we retract $C(\torusn[2n], \tau)$ onto $\Gamma_{2n}$. 
So we reduce the lifting problem to lifting a $*$-$\tau$-homomor-phism $\lambda \colon C(\Gamma_{2n},\tau) \to (Q,\tau)$.

We now use Theorem \ref{thm:oneDimHoles} to open up gaps between notes in $\Gamma_{2n}$, and then we again only make a small mistake when we retract onto the notes. 
This leaves us with the problem of lifting a $*$-$\tau$-homomorphism $\chi \colon C(Y, \tau) \to (Q, \tau)$, where $Y$ is a finite discreet set of points, and where $\tau$ flips the points pairwise and potentially fixes a single point. 
Since semiprojectivity of $C^{*,\tau}$-algebras is closed under direct sums, we can lift $\chi$ as both $(\bbC, \id)$ and $\bbC \oplus \bbC$ with the reflection that flips the two summands are semiprojective $C^{*,\tau}$-algebras, see \cite{LoringSorensenLinsTheorem}. 
\end{proof}

As noted in Remark \ref{rem:anyTwoDim}, we can ``open up holes'' in many two dimensional CW complexes under suitable $K$-theory conditions, we exploit that to prove Theorem \ref{thm:Lin's}.

\begin{proof}[Proof of Theorem~\ref{thm:Lin's}]
The underlying involutive space for problems involving real, normal contractions
is the unit disk with the flip accross the $x$-axis.  After we rephrase this
as a lifting problem, we are free to instead consider $[-1,1]\times [-1,1]$
with the flip accross the $x$-axis.  This is homotopic, as a symmetric space,
to a point, so there is no $K$-theory to worry about.

We consider a grid of points that occur in pairs with one point above, one below, the $x$-axis.  We then are faced with a graph in which most edges
are swapped in pairs by the involution, except that the edges on the $x$-axis
are fixed.  Opening up all the pairs of edges and retracting, we have
reduced to a lifting problem involving the symmetric space consisting of a
line that is fixed and many pairs of points that are swapped.  The corresponding
$C^{*,\tau}$-algebra is easily seen to be semiprojective.
\end{proof}



\end{document}